\documentclass[11pt]{amsart}
\usepackage[utf8]{inputenc}
\usepackage{amsmath,amsfonts,amsthm,amssymb}
\usepackage{enumitem}
\usepackage{graphicx}
\usepackage{pgf,tikz,pgfplots}
\usetikzlibrary{arrows}
\usetikzlibrary{angles,quotes}
\usetikzlibrary{calc}
\usepgfplotslibrary{fillbetween}
\usetikzlibrary{positioning}
\usetikzlibrary{calc,shapes.geometric}
\usetikzlibrary{decorations.pathreplacing}
\usetikzlibrary{cd}
\usetikzlibrary{positioning}
\usepackage{mathrsfs}
\usepackage{caption,subcaption}
\usepackage{mathtools}
\usepackage{xspace}
\usetikzlibrary{patterns}
\calclayout
\usepackage{soul}
\usepackage{csquotes}
\usepackage{comment}
\usepackage{float}

\usepackage{hyperref}
\hypersetup{
  colorlinks   = true,          
  urlcolor     = blue,          
  linkcolor    = green,          
  citecolor   = orange             
}

\newtheorem{theorem}{Theorem}
\newtheorem{corollary}[theorem]{Corollary}

\newtheorem{lemma}[theorem]{Lemma}

\theoremstyle{definition}
\newtheorem{definition}[theorem]{Definition}

\theoremstyle{remark}
\newtheorem{remark}[theorem]{Remark}

\DeclareMathOperator{\conv}{conv}

\newcommand\doi[1]{\href{http://dx.doi.org/#1}{\texttt{doi:#1}}}

\graphicspath{{Figures/}}

\tikzstyle{lattice} = [draw=black, fill=black]
\tikzstyle{valid} = [draw=red, thick, fill=white]
\tikzstyle{intersect} = [draw=orange, fill=orange]
\tikzstyle{boundary} = [draw=blue, fill=blue]
\tikzstyle{triangle} = [draw=black, thick, fill=green!20]
\tikzstyle{inequality} = [draw=green, thick]
\tikzstyle{someline} = [draw=black, dashed]

\title{Characterization of tropical planar curves up to genus six}
\author{Ayush Kumar Tewari}

\address[A.~K.~Tewari]{Department of Mathematics: Algebra and Geometry, Ghent University, Belgium}
\email{ayushkumar.tewari@ugent.be}

\makeatletter
\@namedef{subjclassname@1991}{Mathematics Subject Classification (2020)}
\makeatother



\thanks{%
A.~K.~Tewari has been supported by the UGent BOF grant BOF/STA/201909/038. This research was supported through the programme "Oberwolfach Leibniz Fellows" by the Mathematisches Forschungsinstitut Oberwolfach in 2021}
\pgfplotsset{compat=1.17}
\begin{document}

\maketitle

\begin{abstract}
We provide new forbidden criterion for realizability of smooth tropical plane curves. This in turn provides us a complete classification of smooth tropical plane curves up to genus six.
\end{abstract}

\section{Introduction}

Tropical plane curves exhibit duality with respect to regular subdivisions of a certain lattice polygon. Smooth tropical planar curves are dual to unimodular triangulations. This duality also describes tropical planar curves as metric graphs, where each edge length can be understood as a metric. The skeleton is a trivalent, planar metric graph which allows a deformation retract from the tropical plane curve to itself. It can be obtained via a graph theoretic operation from the tropical curve. An important invariant under this duality is the \emph{genus} of the tropical curve which is equal to the number of cycles in the skeleton and which is also equal to the number of interior lattice points of the corresponding lattice polygon. A central problem concerning skeleta of smooth tropical plane curves has been to categorize which trivalent, planar graphs can occur as skeleta of smooth tropical plane curves.

The graphs which do occur as skeleta of tropical curves are referred as \emph{troplanar} or \emph{tropically planar} \cite{M19}. Starting from the work in \cite{JB15}, there has been immense interest to find forbidden criteria to rule out classes of graphs which can not be realizable. We illustrate the duality between the unimdoular triangulations, tropical plane curves and skeletons by the Figure \ref{fig:double_heavy_cycle_two_loops} and refer the reader to \cite{JB15}, \cite{M19}, \cite{Joswig2020} for further details about the duality and the graph theoretic operation to obtain the skeleton. In \cite{JB15}, computational techniques were employed to classify all troplanar graphs for lower genera $g=3,4$ and $g=5$. This classification provides us a complete classification of all troplanar graphs till genus four. In \cite{M19}, this computational study is pushed to the case of genus six and many new criteria are also established. However the classification still remained incomplete for genus five and six. Some of the previously known forbidden patterns are \emph{sprawling} \cite{JB15}, \emph{crowded} \cite{M17} and \emph{TIE-fighter} \cite{M19}. In \cite{Joswig2020}, for the first time a complete classification of all troplanar graphs up to genus five was achieved which is independent of computational enumeration. Along with completing the case for genus five, \cite{Joswig2020} also provided the list of eight graphs which remained unclassified for the case of genus six, depicted in Figure \ref{fig:genus6_uncategorized}.

\begin{figure}[th]\centering
  \begin{tikzpicture}[scale=0.7]
    \draw[] (0,0) -- (-1,-2);
    \draw[] (-1,-2) -- (-3,-4);
    \draw[] (-6,-4) -- (-3,-4);
    \draw[] (-6,-4) -- (-6,-3);
    \draw[] (-6,-3) -- (0,0);
    \draw[] (0,0) -- (-1,-1);
    \draw[] (-1,-1) -- (-2,-1);
    \draw[] (-1,-1) -- (-1,-2);
    \draw[] (-1,-2) -- (-2,-1);
    \draw[] (-2,-1) -- (-6,-4);
    \draw[] (-6,-4) -- (-5,-3);
    \draw[] (-5,-3) -- (-6,-3);
    \draw[] (-5,-3) -- (-2,-1);
    \draw[] (-1,-2) -- (-6,-4);
    \draw[] (-1,-2) -- (-3,-3);
    \draw[] (-3,-3) -- (-3,-4);
    \draw[] (-3,-3) -- (-6,-4);
    \draw[] (-2,-1) -- (-2,-2);
    \draw[] (-2,-2) -- (-1,-2);
    \draw[] (-2,-2) -- (-3,-2);
    \draw[] (-3,-2) -- (-2,-1);
    \draw[] (-3,-2) -- (-6,-4);
    \draw[] (-3,-2) -- (-4,-3);
    \draw[] (-4,-3) -- (-6,-4);
    \draw[] (-4,-3) -- (-2,-2);
    \draw[] (-4,-3) -- (-1,-2);
    \fill[black] (0,0) circle (.07cm);
    \fill[black] (-1,-1) circle (.07cm);
    \fill[black] (-1,-2) circle (.07cm);
    \fill[black] (-2,-1) circle (.07cm);
    \fill[black] (-2,-2) circle (.07cm);
    \fill[black] (-3,-2) circle (.07cm);
    \fill[black] (-3,-3) circle (.07cm);
    \fill[black] (-4,-3) circle (.07cm);
    \fill[black] (-6,-4) circle (.07cm);
    \fill[black] (-6,-3) circle (.07cm);
    \fill[black] (-5,-3) circle (.07cm);
  \end{tikzpicture}
  \qquad
  \begin{tikzpicture}[scale=0.165]
    \draw[blue] (0,0) -- (0,2);
    \draw[blue] (0,0) -- (2,0);
    \draw[blue] (0,2) -- (2,0);
    \draw[] (0,2) -- (-0.5,4);
    \draw[] (2,0) -- (4,-0.5);
    \draw[brown] (0,0) -- (-1,-1);
    \draw[violet] (-1,-1) -- (-1,-5);
    \draw[violet] (-1,-1) -- (-3,-1);
    \draw[magenta] (-3,-1) -- (-3,-2);
    \draw[violet] (-3,-2) -- (-1,-5);
    \draw[magenta] (-3,-2) -- (-5,0);
    \draw[magenta] (-3,-1) -- (-7,3);
    \draw[magenta] (-5,0) -- (-7,3);
    \draw[olive] (-1,-5) -- (1,-11);
    \draw[olive] (1,-11) -- (-5,0);
    \draw[teal] (-7,3) -- (-8.5,5);
    \draw[red] (-8.5,5) -- (-12.5,11);
    \draw[red] (-8.5,5) -- (-12.5,9);
    \draw[red] (-12.5,9) -- (-12.5,11);
    \draw[] (-12.5,9) -- (-13.5,9);
    \draw[] (-12.5,11) -- (-13.5,13);
    \draw[orange] (1,-11) -- (2,-13.5);
    \draw[green] (2,-13.5) -- (4,-19.5);
    \draw[green] (2,-13.5) -- (5,-18.5);
    \draw[green] (4,-19.5) -- (5,-18.5);
    \draw[] (4,-19.5) -- (4,-21);
     \draw[] (5,-18.5) -- (7,-19.5);
  \end{tikzpicture}
  \qquad
  \begin{tikzpicture}[scale=0.35]
  \centering
   \draw[magenta] (-2,0) -- (0,0);
   \draw[violet] (0,0) -- (2,0);
   \draw[magenta] (-2,0) -- (-2,-2);
   \draw[magenta] (-2,-2) -- (0,-2);
   \draw[violet] (0,-2) -- (2,-2);
   \draw[violet] (2,-2) -- (2,0);
   \draw[magenta] (0,0) -- (0,-2);
   \draw[teal] (-2,0) -- (-3,1);
   \draw[brown] (2,0) -- (3,1);
   \draw[olive] (-2,-2) -- (0,-4);
   \draw[olive] (2,-2) -- (0,-4);
   \draw[orange] (0,-4) -- (0,-5);

    \filldraw[fill=white,draw=red] (-3,1) circle (0.5cm);
    \filldraw[fill=white,draw=blue] (3,1) circle (0.5cm);
        \filldraw[fill=white,draw=green] (0,-5.5) circle (0.5cm);

  \end{tikzpicture}
  \caption{A unimodular triangulation of a genus~6 polytope (left), its dual graph (center), and the corresponding skeleton which has a double heavy cycle with two loops (right)}
  \label{fig:double_heavy_cycle_two_loops}
\end{figure}
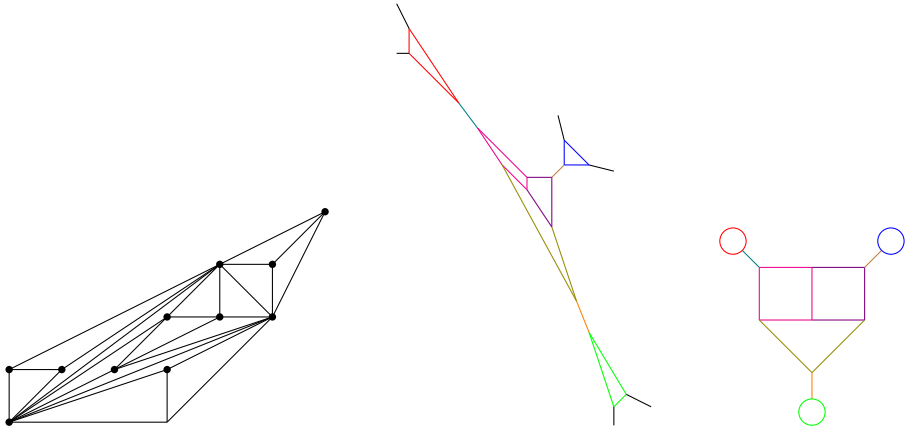

In this article we provide the first complete classification of all troplanar graphs up to genus six in the form of Theorem \ref{thm:classify_genus_six}. We provide a new criterion stated as \emph{heavy cycle with one loop}, which extends the arguments of the criterion \emph{heavy cycle with two loops} discussed in \cite{Joswig2020}. We also define a new graph structure called \emph{double heavy cycle} and provide a structural result for a graph with double heavy cycle to be troplanar, which helps us in completely classifying all genus six graphs. 

\begin{figure}[H]\centering
  \begin{tikzpicture}[scale=0.4]
    \draw[] (0,0) -- (2,0);
    \draw[] (0,0) -- (0,-2);
    \draw[] (0,-2) -- (2,-2);
    \draw[] (2,0) -- (1.5,-1);
    \draw[] (2,0) -- (2.5,-1);
    \draw[] (1.5,-1) -- (2.5,-1);
    \draw[] (1.5,-1) -- (2,-2);
    \draw[] (2.5,-1) -- (2,-2);
    \draw[] (0,0) -- (-1,0);
    \draw[] (-1,0) -- (-2,-1);
    \draw[] (-1,0) -- (-2,1);
    \draw[] (0,-2) -- (-1,-2);
    \draw (-1.5,-2) circle (0.5cm);
    \draw (-2,0) ellipse (0.2cm and 1cm);
    \node[draw] at (0,-3) {a};
    \fill[black] (0,0) circle (.1cm) node[align=left,   above]{};
    \fill[black] (0,-2) circle (.1cm) node[align=left,   above]{};
    \fill[black] (2,0) circle (.1cm) node[align=left,   above]{};
    \fill[black] (-1,0) circle (.1cm) node[align=left,   above]{};
    \fill[black] (-2,-1) circle (.1cm) node[align=left,   above]{};
    \fill[black] (-2,1) circle (.1cm) node[align=left,   above]{};
    \fill[black] (-1,-2) circle (.1cm) node[align=left,   above]{};
    \fill[black] (0,-2) circle (.1cm) node[align=left,   above]{};
    \fill[black] (2,-2) circle (.1cm) node[align=left,   above]{};
    \fill[black] (1.5,-1) circle (.1cm) node[align=left,   above]{};
    \fill[black] (2.5,-1) circle (.1cm) node[align=left,   above]{};
  \end{tikzpicture}
  \qquad
  \begin{tikzpicture}[scale=0.4]
    \draw[] (15,0) -- (16,0);
    \draw[] (13,0) -- (13,-2);
    \draw[] (15,-2) -- (16,-2);
    \draw[] (16,0) -- (16,-2);
    \draw[] (16,-2) -- (17,-2);
    \draw[] (16,0) -- (17,0);
    \draw[] (17,0) -- (18,1);
    \draw[] (17,0) -- (18,-1);
    \draw (17.5,-2) circle (0.5cm);
    \draw (14,0) ellipse (1cm and 0.3cm);
    \draw (14,-2) ellipse (1cm and 0.3cm);
    \draw (18,0) ellipse (0.2cm and 1cm);
    \node[draw] at (16,-3) {b};
    \fill[black] (13,0) circle (.1cm) node[align=left,   above]{};
    \fill[black] (13,-2) circle (.1cm) node[align=left,   above]{};
    \fill[black] (15,0) circle (.1cm) node[align=left,   above]{};
    \fill[black] (15,-2) circle (.1cm) node[align=left,   above]{};
    \fill[black] (16,0) circle (.1cm) node[align=left,   above]{};
    \fill[black] (16,-2) circle (.1cm) node[align=left,   above]{};
    \fill[black] (17,0) circle (.1cm) node[align=left,   above]{};
    \fill[black] (18,1) circle (.1cm) node[align=left,   above]{};
    \fill[black] (18,-1) circle (.1cm) node[align=left,   above]{};
    \fill[black] (17,-2) circle (.1cm) node[align=left,   above]{};
  \end{tikzpicture}
  \qquad
  \vspace{0.3cm}
  \begin{tikzpicture}[scale=0.4]
    \draw[] (0,-4) -- (2,-4);
    \draw[] (0,-4) -- (0,-6);
    \draw[] (0,-6) -- (2,-6);
    \draw[] (0,-4) -- (-2,-4);
    \draw[] (0,-6) -- (-2,-6);
    \draw[] (2,-4) -- (2,-6);
    \draw[] (2,-4) -- (3,-4);
    \draw[] (2,-6) -- (3,-6);
    \draw[] (5,-6) -- (6,-6);
    \draw (3.5,-4) circle (0.5cm);
    \draw (6.5,-6) circle (0.5cm);
    \draw (-2,-5) ellipse (0.2cm and 1cm);
    \draw (4,-6) ellipse (1cm and 0.2cm);
    \node[draw] at (0,-7) {c};
    \fill[black] (0,-4) circle (.1cm) node[align=left,   above]{};
    \fill[black] (0,-6) circle (.1cm) node[align=left,   above]{};
    \fill[black] (-2,-6) circle (.1cm) node[align=left,   above]{};
    \fill[black] (2,-6) circle (.1cm) node[align=left,   above]{};
    \fill[black] (3,-6) circle (.1cm) node[align=left,   above]{};
    \fill[black] (5,-6) circle (.1cm) node[align=left,   above]{};
    \fill[black] (2,-4) circle (.1cm) node[align=left,   above]{};
    \fill[black] (-2,-4) circle (.1cm) node[align=left,   above]{};
    \fill[black] (3,-4) circle (.1cm) node[align=left,   above]{};
  \end{tikzpicture}
  \qquad
  \begin{tikzpicture}[scale=0.4]
    \draw[] (12,-4) -- (12,-6);
    \draw[] (12,-6) -- (11,-6);
    \draw[] (12,-4) -- (13,-4);
    \draw[] (12,-6) -- (13,-6);
    \draw[] (13,-4) -- (13,-6);
    \draw[] (13,-4) -- (14,-5);
    \draw[] (13,-6) -- (14,-5);
    \draw[] (14,-5) -- (15,-5);
    \draw[] (11,-4) -- (12,-4);
    \draw[] (9,-4) -- (8,-4);
    \draw (10,-4) ellipse (1cm and 0.2cm);
    \draw (7.5,-4) circle (0.5cm);
    \draw (10.5,-6) circle (0.5cm);
    \draw (15.5,-5) circle (0.5cm);
    \node[draw] at (12,-7) {d};
    \fill[black] (12,-4) circle (.1cm) node[align=left,   above]{};
    \fill[black] (11,-4) circle (.1cm) node[align=left,   above]{};
    \fill[black] (9,-4) circle (.1cm) node[align=left,   above]{};
    \fill[black] (8,-4) circle (.1cm) node[align=left,   above]{};
    \fill[black] (13,-4) circle (.1cm) node[align=left,   above]{};
    \fill[black] (12,-6) circle (.1cm) node[align=left,   above]{};
    \fill[black] (11,-6) circle (.1cm) node[align=left,   above]{};
    \fill[black] (13,-6) circle (.1cm) node[align=left,   above]{};
    \fill[black] (14,-5) circle (.1cm) node[align=left,   above]{};
    \fill[black] (15,-5) circle (.1cm) node[align=left,   above]{};
  \end{tikzpicture}
  \qquad
  \begin{tikzpicture}[scale=0.4]
    \draw[] (1,-9.5) -- (1,-11.5);
    \draw[] (1,-9.5) -- (-2,-9.5);
    \draw[] (-2,-11.5) -- (-2,-9.5);
    \draw[] (1,-9.5) -- (4,-9.5);
    \draw[] (4,-9.5) -- (4,-11.5);
    \draw[] (1,-11.5) -- (0,-11.5);
    \draw[] (1,-11.5) -- (2,-11.5);
    \draw[] (-2,-9.5) -- (-2,-9);
    \draw[] (4,-9.5) -- (4,-9);
    \draw (4,-8.5) circle (0.5cm);
    \draw (-2,-8.5) circle (0.5cm);
    \draw (-1,-11.5) ellipse (1cm and 0.2cm);
    \draw (3,-11.5) ellipse (1cm and 0.2cm);
    \node[draw] at (1,-12.5) {e};
    \fill[black] (1,-9.5) circle (.1cm) node[align=left,   above]{};
    \fill[black] (-2,-9.5) circle (.1cm) node[align=left,   above]{};
    \fill[black] (4,-9.5) circle (.1cm) node[align=left,   above]{};
    \fill[black] (1,-11.5) circle (.1cm) node[align=left,   above]{};
    \fill[black] (0,-11.5) circle (.1cm) node[align=left,   above]{};
    \fill[black] (2,-11.5) circle (.1cm) node[align=left,   above]{};
    \fill[black] (-2,-9) circle (.1cm) node[align=left,   above]{};
    \fill[black] (4,-9) circle (.1cm) node[align=left,   above]{};
    \fill[black] (-2,-11.5) circle (.1cm) node[align=left,   above]{};
    \fill[black] (4,-11.5) circle (.1cm) node[align=left,   above]{};
  \end{tikzpicture}
  \qquad
  \begin{tikzpicture}[scale=0.4]
    \draw[] (9,-9) -- (11,-9);
    \draw[] (9,-11) -- (11,-11);
    \draw[] (11,-9) -- (10,-10);
    \draw[] (10,-10) -- (11,-11);
    \draw[] (10,-10) -- (11,-10);
    \draw[] (11,-9) -- (14,-9);
    \draw[] (11,-11) -- (14,-11);
    \draw[] (14,-9) -- (14,-11);
    \draw[] (11,-10) -- (14,-9);
    \draw[] (11,-10) -- (14,-11);
    \draw[] (12,-9.66) -- (12,-10.34);
    \draw (8.5,-9) circle (0.5cm);
    \draw (8.5,-11) circle (0.5cm);
    \node[draw] at (10,-12) {f};
    \fill[black] (9,-9) circle (.1cm) node[align=left,   above]{};
    \fill[black] (11,-9) circle (.1cm) node[align=left,   above]{};
    \fill[black] (10,-10) circle (.1cm) node[align=left,   above]{};
    \fill[black] (9,-11) circle (.1cm) node[align=left,   above]{};
    \fill[black] (11,-11) circle (.1cm) node[align=left,   above]{};
    \fill[black] (11,-10) circle (.1cm) node[align=left,   above]{};
    \fill[black] (12,-9.66) circle (.1cm) node[align=left,   above]{};
    \fill[black] (12,-10.34) circle (.1cm) node[align=left,   above]{};
    \fill[black] (14,-9) circle (.1cm) node[align=left,   above]{};
    \fill[black] (14,-11) circle (.1cm) node[align=left,   above]{};
  \end{tikzpicture}
  \qquad
  \begin{tikzpicture}[scale=0.4]
    \draw[] (0,-15) -- (-1,-15);
    \draw[] (0,-15) -- (1,-14);
    \draw[] (0,-15) -- (1,-16);
    \draw[] (1,-14) -- (3,-15);
    \draw[] (1,-16) -- (3,-15);
    \draw[] (2,-15.5) -- (2,-14.5);
    \draw[] (1,-14) -- (5,-14);
    \draw[] (1,-16) -- (5,-16);
    \draw[] (3,-15) -- (4,-15);
    \draw[] (4,-15) -- (5,-14);
    \draw[] (4,-15) -- (5,-16);
    \draw[] (5,-14) -- (5,-16);
    \draw (-1.5,-15) circle (0.5cm);
    \node[draw] at (2,-17.5) {g};
    \fill[black] (1,-14) circle (.1cm) node[align=left,   above]{};
    \fill[black] (5,-14) circle (.1cm) node[align=left,   above]{};
    \fill[black] (1,-16) circle (.1cm) node[align=left,   above]{};
    \fill[black] (5,-16) circle (.1cm) node[align=left,   above]{};
    \fill[black] (0,-15) circle (.1cm) node[align=left,   above]{};
    \fill[black] (-1,-15) circle (.1cm) node[align=left,   above]{};
    \fill[black] (2,-14.5) circle (.1cm) node[align=left,   above]{};
    \fill[black] (2,-15.5) circle (.1cm) node[align=left,   above]{};
    \fill[black] (3,-15) circle (.1cm) node[align=left,   above]{};
    \fill[black] (4,-15) circle (.1cm) node[align=left,   above]{};
  \end{tikzpicture}
  \qquad
  \begin{tikzpicture}[scale=0.4]
    \draw[] (8,-16) -- (9,-16);
    \draw[] (9,-16) -- (10,-15.5);
    \draw[] (9,-16) -- (10,-16.5);
    \draw[] (10,-15.5) -- (10,-16.5);
    \draw[] (10,-16.5) -- (13,-16.5);
    \draw[] (10,-15.5) -- (11.5,-14);
    \draw[] (11.5,-13) -- (11.5,-14);
    \draw[] (13,-15.5) -- (11.5,-14);
    \draw[] (13,-15.5) -- (13,-16.5);
    \draw[] (13,-15.5) -- (14,-16);
    \draw[] (13,-16.5) -- (14,-16);
    \draw[] (14,-16) -- (15,-16);
    \draw (7.5,-16) circle (0.5cm);
    \draw (15.5,-16) circle (0.5cm);
    \draw (11.5,-12.5) circle (0.5cm);
    \node[draw] at (12,-17.5) {h};
    \fill[black] (8,-16) circle (.1cm) node[align=left,   above]{};
    \fill[black] (9,-16) circle (.1cm) node[align=left,   above]{};
    \fill[black] (10,-15.5) circle (.1cm) node[align=left,   above]{};
    \fill[black] (10,-16.5) circle (.1cm) node[align=left,   above]{};
    \fill[black] (11.5,-14) circle (.1cm) node[align=left,   above]{};
    \fill[black] (11.5,-13) circle (.1cm) node[align=left,   above]{};
    \fill[black] (13,-15.5) circle (.1cm) node[align=left,   above]{};
    \fill[black] (13,-16.5) circle (.1cm) node[align=left,   above]{};
    \fill[black] (14,-16) circle (.1cm) node[align=left,   above]{};
    \fill[black] (15,-16) circle (.1cm) node[align=left,   above]{};
  \end{tikzpicture}
  \caption{The eight trivalent planar graphs of genus six, which are not realizable \cite{M19}, and remained unclassified up till now.}
  \label{fig:genus6_uncategorized}
\end{figure}
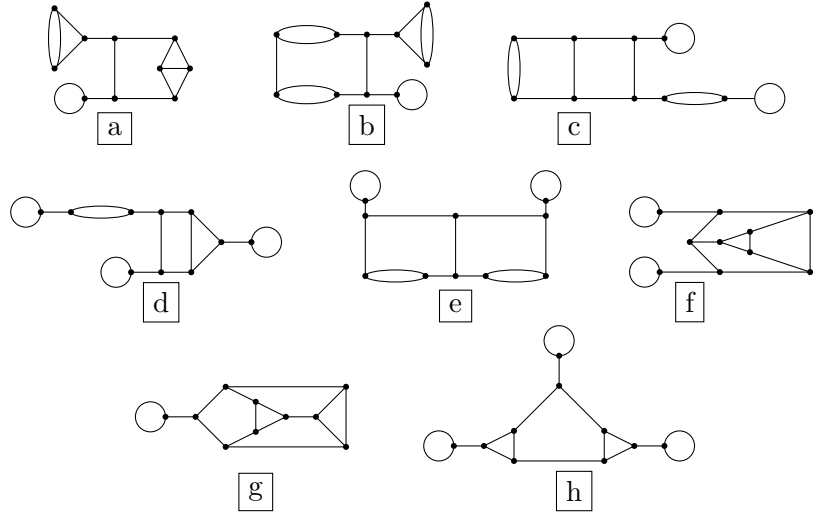

\section{Heavy cycle with one loop}

Let $P$ be a lattice polygon and $\partial P$ represent the boundary of $P$. A lattice polygon with all interior lattice points in a line is called a \emph{hyperelliptic} polygon. We refer to the convex hull of the interior lattice points of a lattice polygon $P$ as $\text{int}(P)$. We recall the following lemma from \cite{Joswig2020},

\begin{lemma}\label{lem:abz}
  Let $P$ contain a unimodular triangle face with vertices $a,b,z$ such that neither $a$ nor $b$ is a vertex of $P$, and $z$ is an interior lattice point.
  If $a$ and $b$ lie on $\partial P$ then either $a$ and $b$ lie on a common edge of $P$ or the lattice point $a+b-z$ is contained in $P$.
\end{lemma}

We state the following definitions of a \emph{heavy cycle} and a \emph{heavy cycle with two loops} from \cite{Joswig2020},

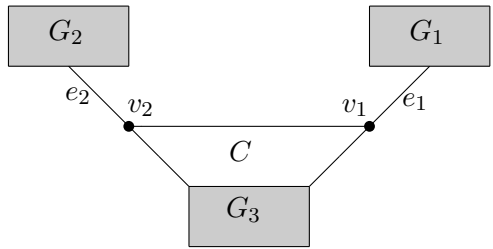
\begin{figure}[th]\centering
  \begin{tikzpicture}[scale=0.8]
    \fill[gray!40!white] (0.5,-0.5) rectangle (2.5,-1.5);
    \fill[gray!40!white] (-2.5,2.5) rectangle (-0.5,1.5);
    \fill[gray!40!white] (3.5,2.5) rectangle (5.5,1.5);
    \draw[] (0,0) -- (-0.5,0.5);
    \draw[] (0,0) -- (0.5,-0.5);
    \draw[] (-0.5,0.5) -- (3.5,0.5);
    \draw[] (-1.5,1.5) -- (-0.5,0.5);
    \draw[] (3.5,0.5) -- (4.5,1.5);
    \draw[] (0.5,-0.5) -- (2.5,-0.5);
    \draw[] (0.5,-0.5) -- (0.5,-1.5);
    \draw[] (2.5,-0.5) -- (2.5,-1.5);
    \draw[] (2.5,-1.5) -- (0.5,-1.5);
    \draw[] (2.5,-0.5) -- (3,0);
    \draw[] (3,0) -- (3.5,0.5);
    \draw[] (5.5,1.5) -- (3.5,1.5);
    \draw[] (5.5,2.5) -- (3.5,2.5);
    \draw[] (5.5,1.5) -- (5.5,2.5);
    \draw[] (3.5,1.5) -- (3.5,2.5);
    \draw[] (-2.5,1.5) -- (-0.5,1.5);
    \draw[] (-2.5,2.5) -- (-0.5,2.5);
    \draw[] (-2.5,1.5) -- (-2.5,2.5);
    \draw[] (-0.5,1.5) -- (-0.5,2.5);
    \fill[black] (-0.5,0.5) circle (.09cm) node[align=right, above]{};
    \fill[black] (-0.3,0.5) circle (.000001cm) node[align=right, above]{$v_{2}$};
    \fill[black] (3.5,0.5) circle (.09cm) node[align=right,   above]{$v_{1}\quad$};
    \fill[black] (1.6,-0.25) circle (.0001cm) node[align=left,   above]{$C\quad$};
    \fill[black] (-1.1,0.7) circle (.0001cm) node[align=left,   above]{$e_{2}\quad$};
    \fill[black] (4.5,0.6) circle (.0001cm) node[align=left,   above]{$e_{1}\quad$};
    \fill[black] (1.65,-1.25) circle (.0001cm) node[align=left,   above]{$G_{3}\quad$};
    \fill[black] (-1.3,1.7) circle (.0001cm) node[align=left,   above]{$G_{2}\quad$};
    \fill[black] (4.7,1.7) circle (.0001cm) node[align=left,   above]{$G_{1}\quad$};
  \end{tikzpicture}
  \caption{Graph with the heavy cycle $C$}
  \label{fig:heavy}
\end{figure}

\begin{definition}\label{def:heavy}
  We say that a cycle $C$ in a planar graph $G$ is \emph{heavy}, if
  \begin{enumerate}
  \item it has two nodes, $v_1$ and $v_2$, such that $v_i$ is incident with a cut edge $e_i$ connecting $v_i$ with a subgraph $G_i$ of positive genus;
  \item and there is a third subgraph, $G_3$, also of positive genus, which shares at least one node with the cycle $C$; cf.\ Figure~\ref{fig:heavy}.

  \end{enumerate}
\end{definition}

We recall some basic notations from \cite{Joswig2020} concerning heavy cycle. From \cite[Lemma 1]{Joswig2020} we know that there are split lines, $S_1$ and $S_2$, dual to the edges $e_1$ and $e_2$ of $G$. Also by \cite[Lemma 4]{Joswig2020}, $P$ is decomposed into a union of three lattice polygons $P_1$, $P_2$ and $P_{3}$ such that $\Delta$ induces triangulations of all three. We obtain triangulations $\Delta_1$, $\Delta_2$ and $\Delta_{3}$ such that the component $G_i$ is the skeleton of $\Delta_i$ for $i=1,2,$ and $G_{3} \cup C$ is realized by $\Delta_{3}$. $T_{1}$ and $T_{2}$ represent the triangular faces in $\Delta$ dual to $v_{1}$ and $v_{2}$ respectively. The polygon $P_{3}$ is referred as the \emph{heavy component} of $P$, and $\Delta_{3}$ is referred as the \emph{heavy component} of $\Delta$.

In \cite{Joswig2020} a structural result concerning heavy cycles is proved specifying the conditions under which a graph with a heavy cycle is realizable.

\begin{lemma}\label{lem:heavy}
  Suppose that $G$ has a heavy cycle with cut edges $e_1$ and $e_2$ as in Figure~\ref{fig:heavy}.
  Then the triangles $T_1$ and $T_2$ in $\Delta$ share an edge $[z,w]$, where $z$ is the interior lattice point dual to $C$, and the split lines $S_1$ and $S_2$ intersect in $w$, which is a vertex of $P_{3}$, and which lies in the boundary of $P$.
\end{lemma}

\begin{definition}
  We say that a connected trivalent planar graph $G$ has a \emph{heavy cycle with two loops} if it has the form as described in Figure \ref{fig:heavy-two-loops}, where $G_{3}$ represents a subgraph of positive genus.	
\end{definition} 
 
\begin{figure}[ht]\centering
  \begin{tikzpicture}[scale=0.9]
    \coordinate (v2) at (-0.5,0.5);
	\coordinate (v1) at (3.5,0.5);
	\coordinate (u2) at (-1.5,1.5);
	\coordinate (u1) at (4.5,1.5);
	
	\fill[gray!40!white] (0.5,-0.5) rectangle (2.5,-1);
	\draw[] (0,0) -- (-0.5,0.5);
	\draw[] (-0.5,0.5) -- (3.5,0.5);
	\draw[] (0,0) -- (0.5,-0.5);
	\draw[] (-1.5,1.5) -- (-0.5,0.5);
	\draw[] (3.5,0.5) -- (4.5,1.5);
	\draw[] (-1.85,1.85) circle (0.5cm);
	\draw[] (0.5,-0.5) -- (2.5,-0.5);
	\draw[] (0.5,-0.5) -- (0.5,-1);
	\draw[] (2.5,-0.5) -- (2.5,-1);
	\draw[] (2.5,-1) -- (0.5,-1);
	\draw[] (2.5,-0.5) -- (3,0);
	\draw[] (3,0) -- (3.5,0.5);
	\draw[] (4.85,1.85) circle (0.5cm);
	
    \node at ($(v2)$) [above right=-0.1pt] {$v_{2}$};
	\node at ($(v1)$) [above left=-0.1pt] {$v_{1}$};
	
	\fill[black] (-0.5,0.5) circle (.09cm);
	\fill[black] (3.5,0.5) circle (.09cm);
	\fill[black] (-1.25,0.7) circle (.0001cm) node[align=left,   above]{$e_{2}$};
	\fill[black] (4.3,0.7) circle (.0001cm) node[align=left,   above]{$e_{1}$};
	\fill[black] (4.5,1.5) circle (.1cm) node[align=left,   above]{};
	\fill[black] (-1.5,1.5) circle (.1cm) node[align=left,   above]{};
	\fill[black] (1.65,-0.2) circle (.0001cm) node[align=left,   above]{$C\quad$};
	\fill[black] (1.65,-1.05) circle (.0001cm) node[align=left,   above]{$G_{3}\quad$};
	\end{tikzpicture}
	\caption{Heavy cycle with two loops}
	\label{fig:heavy-two-loops}
\end{figure}
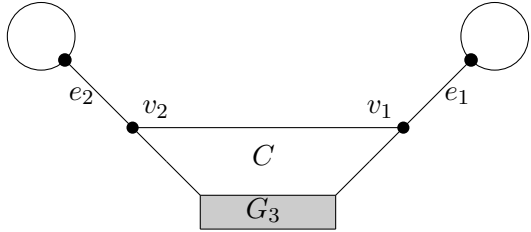

Using Lemma \ref{lem:heavy} in \cite{Joswig2020}, the following forbidden criterion is established,

\begin{theorem}\label{thm:heavy-two-loops}
  Suppose $G$ is a graph with a heavy cycle $C$ and two loops with cut edges, $e_{1}$ and $e_{2}$, as in Figure \ref{fig:heavy-two-loops}.
  Then the heavy component $P_{3}$ can have at most three interior lattice points, and these lie on the line spanned by the edge $[z,w] \in \Delta$, where $z$ is the interior lattice point dual to $C$, and $w$ is the intersection point of the split edges $s_{1}$ and $s_{2}$.
  In particular, $P_{3}$ is hyperelliptic and $g\leq 5$.
\end{theorem}

Inspired from these definitions, we define a graph with a \emph{heavy cycle with one loop},

\begin{definition}
We say that a planar, trivalent graph $G$ has a \emph{heavy cycle with one loop} if it is of the form shown in Figure \ref{fig:heavy-one-loop}, where $G_{2}$ and $G_{3}$ represent subgraphs of positive genus.
\end{definition}

\begin{figure}[H]\centering
  \begin{tikzpicture}[scale=1.1]
    \coordinate (v2) at (-0.5,0.5);
	\coordinate (v1) at (3.5,0.5);
	\coordinate (u2) at (-1.5,1.5);
	\coordinate (u1) at (4.5,1.5);
	
	\fill[gray!40!white] (0.5,-0.5) rectangle (2.5,-1);
	\fill[gray!40!white] (-2,2.5) rectangle (-1.5,1.5);
	\draw[] (0,0) -- (-0.5,0.5);
	\draw[] (-0.5,0.5) -- (3.5,0.5);
	\draw[] (0,0) -- (0.5,-0.5);
	\draw[] (-1.5,1.5) -- (-0.5,0.5);
	\draw[] (3.5,0.5) -- (4.5,1.5);
	\draw[] (0.5,-0.5) -- (2.5,-0.5);
	\draw[] (0.5,-0.5) -- (0.5,-1);
	\draw[] (2.5,-0.5) -- (2.5,-1);
	\draw[] (2.5,-1) -- (0.5,-1);
	\draw[] (2.5,-0.5) -- (3,0);
	\draw[] (3,0) -- (3.5,0.5);
	\draw[] (-1.5,1.5) -- (-2,1.5);
	\draw[] (-2,1.5) -- (-2,2.5);
	\draw[] (-2,2.5) -- (-1.5,2.5);
	\draw[] (-1.5,1.5) -- (-1.5,2.5);
	\draw[] (4.85,1.85) circle (0.5cm);
	
    \node at ($(v2)$) [above right=-0.1pt] {$v_{2}$};
	\node at ($(v1)$) [above left=-0.1pt] {$v_{1}$};
	
	\fill[black] (-0.5,0.5) circle (.09cm);
	\fill[black] (3.5,0.5) circle (.09cm);
	\fill[black] (-1.25,0.7) circle (.0001cm) node[align=left,   above]{$e_{2}$};
	\fill[black] (4.3,0.7) circle (.0001cm) node[align=left,   above]{$e_{1}$};
	\fill[black] (4.5,1.5) circle (.1cm) node[align=left,   above]{};
	\fill[black] (1.65,-0.2) circle (.0001cm) node[align=left,   above]{$C\quad$};
	\fill[black] (1.65,-1.05) circle (.0001cm) node[align=left,   above]{$G_{3}\quad$};
	\fill[black] (-1.55,1.75) circle (.0001cm) node[align=left,   above]{$G_{2}\quad$};
	\end{tikzpicture}
	\caption{Heavy cycle with one loop}
	\label{fig:heavy-one-loop}
\end{figure}
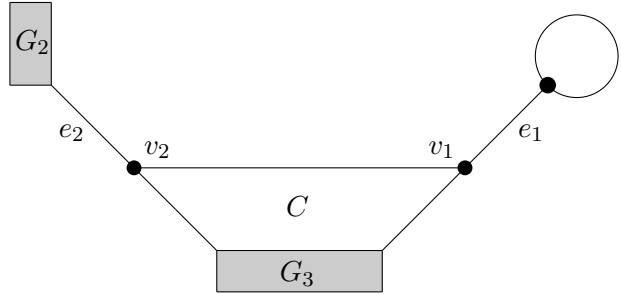

Similar to Theorem \ref{thm:heavy-two-loops} we prove the first new forbidden criterion in this article, along with additional conditions in the specific case of $g = 6$,

\begin{theorem}\label{thm:heavy-one-loop}
Suppose $G$ is a tropically planar graph with a heavy cycle with one loop as shown in Figure \ref{fig:heavy-one-loop}, then the heavy component $P_{3}$ is hyperelliptic and can have at most three interior lattice points. Also, $P_{2}$ can have at most three interior lattice points. In the case when genus $g = 6$ and $g(P_{2}) = 2$, $P_{2}$ is hyperelliptic and the triangulation restricted to $P_{2}$, i.e, $\Delta_{2}$ cannot have a nontrivial split. In particular, genus of $G$ can be at most seven. 
\end{theorem}

\begin{proof}
  \begin{figure}\centering
    \begin{tikzpicture}[scale=0.75]
      \coordinate (z) at (0,0);
      \coordinate (p1) at (1,0);
      \coordinate (q1) at (0,1);
      \coordinate (p2) at (0,1);
      \coordinate (q2) at (-1,-2);
      \coordinate (u1) at (1,1);
      \coordinate (v1) at (-1,-1);
      \coordinate (k1) at (0,-1);
      
      \filldraw[triangle] (z) -- (p1) -- (q1) -- cycle;
      \filldraw[triangle] (z) -- (p2) -- (q2) -- cycle;
      
      \node[above right] at (z) {$T_1$};
      \filldraw[lattice] (z) circle (1.5pt);
      \draw[] (0,0) -- (0,1);
      \draw[] (0,0) -- (1,0);
      \draw[] (0,1) -- (1,0);
      \draw[] (0,1) -- (-1,-2);
      \draw[] (0,0) -- (-1,-2);
      \draw[dashed] (0,1) -- (2,1);
      \draw[dashed] (1,0) -- (2,0);
      \draw[dashed] (-1,-2) -- (-1,-3);
      \draw[dashed] (1,0) -- (1,-3);
      \draw[dashed] (1.4,1.5) -- (-0,-3.8);
      \draw[dashed] (-1.5,-1) -- (0,-4);
      \fill[black] (0.05,0) circle (.000005cm) node[align=right,   below]{$z$};
      \fill[black] (0,1) circle (.000005cm) node[align=right,   above]{$w$};
      \fill[black] (1,1) circle (.000005cm) node[align=right,   above]{$r$};
      \fill[black] (0.7,0.3) circle (.000005cm) node[align=right,   above]{$s_{1}$};
      \fill[black] (-0.65,-0.5) circle (.000005cm) node[align=right,   above]{$s_{2}$};
      \fill[black] (-1,-1) circle (.000005cm) node[align=right,   above]{$r'$};
      \fill[black] (0.2,-1.2) circle (.000005cm) node[align=right,   below]{$P_{3}$};
      \fill[black] (-1.2,-1.9) circle (.000005cm) node[align=left,   below]{$p_2$};
      \fill[black] (1.25,-0.45) circle (.0000005cm) node[align=left,   above]{$p_1$};
      \foreach \x in {-2,...,2}{
        \foreach \y in {-4,...,3}{
          \filldraw[lattice] (\x,\y) circle (0.6pt);
        }
      }
      \filldraw[boundary] (p1) circle (1.5pt);
      \filldraw[boundary] (q1) circle (1.5pt);
      \filldraw[boundary] (p2) circle (1.5pt);
      \filldraw[boundary] (q2) circle (1.5pt);
      \filldraw[lattice] (u1) circle (1.5pt);
      \filldraw[lattice] (v1) circle (1.5pt);
      \filldraw[lattice] (k1) circle (1.5pt);
    \end{tikzpicture}
    \hspace{1cm}
    \begin{tikzpicture}[scale=0.75]
    \coordinate (z) at (0,0);
    \coordinate (p1) at (1,0);
    \coordinate (q1) at (0,1);
    \coordinate (p2) at (-1,-5);
    \coordinate (q2) at (0,1);
    \coordinate (u1) at (1,1);
    \coordinate (v1) at (-1,-4);
    \coordinate (k1) at (0,-1);
    \coordinate (k2) at (0,-2);
    \coordinate (k3) at (0,-3);
    
    \filldraw[triangle] (z) -- (p1) -- (q1) -- cycle;
    
    \node[above right] at (z) {$T_1$};
    \filldraw[lattice] (z) circle (1.5pt);
    \draw[] (0,0) -- (0,1);
    \draw[] (0,0) -- (1,0);
    \draw[] (0,1) -- (1,0);
    \draw[dashed] (1,0) -- (-0.25,-3.75);
    \draw[dashed] (1,0) -- (2.25,3.75);
    \draw[dashed] (0,1) -- (2.5,3.5);
    \fill[black] (0,0) circle (.000005cm) node[align=right,   below]{$z$};
    \fill[black] (0,1) circle (.000005cm) node[align=right,   above]{$w$};
    \fill[black] (1,1) circle (.000005cm) node[align=right,   above]{$r$};
    \fill[black] (1,2) circle (.000005cm) node[align=right,   above]{$q$};
    \fill[black] (0.7,0.3) circle (.000005cm) node[align=right,   above]{$s_{1}$};
    \fill[black] (1.3,-0.2) circle (.0000005cm) node[align=left,   above]{$p_1$};
    \foreach \x in {-2,...,2}{
    	\foreach \y in {-4,...,3}{
    		\filldraw[lattice] (\x,\y) circle (0.6pt);
    	}
    }
    \filldraw[boundary] (p1) circle (1.5pt);
    \filldraw[boundary] (q1) circle (1.5pt);
    \filldraw[lattice] (u1) circle (1.5pt);
    \filldraw[lattice] (k1) circle (1.5pt);
    \filldraw[lattice] (k2) circle (1.5pt);
    \filldraw[lattice] (k3) circle (1.5pt);
    \end{tikzpicture}
    \caption{This illustrates Theorem \ref{thm:heavy-one-loop}: general sketch (left) and the case when $g(P') \geq 4$ (right), which is impossible}
    \label{fig:cycle_two_loops}
  \end{figure}
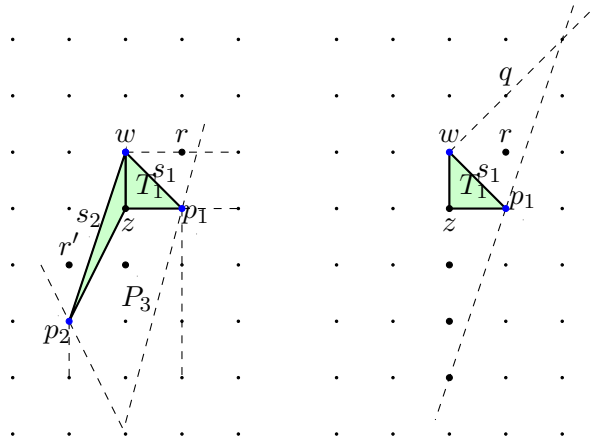

Let $T_{1} = \conv\{(0,0),(1,0),(0,1)\}$ and $T_{2}= \conv\{(0,0),(-1,-k),(0,1)\}$ be the
triangles dual to $v_{1}$ and $v_{2}$ respectively in $\Delta$ with the edges $s_{1}$(between the points $(1,0)$ and $(0,1)$) and $s_{2}$(between the points $(0,1)$ and $(-1,-k)$) being the split edges corresponding to the cut edges $e_{1}$ and $e_{2}$ in Figure \ref{fig:heavy-one-loop}. We let $z=(0,0)$ be the point corresponding to the heavy cycle $C$ and by the heavy cycle lemma we know that $T_{1}$ and $T_{2}$ share an edge and the split edges intersect at a point, in this case that point is $w = (0,1)$. This is illustrated in Figure ~\ref{fig:cycle_two_loops}. Invoking Lemma ~\ref{lem:abz} for $T_{1}$ we realize that the point $r = (1,1)$ is in $P$. If we consider the case that $r$ is in $\partial P$, that implies that the sub polygon of $P$ realizing the loop lies in between the parallel lines $y=0$ and $y=1$, which gives a contradiction as this does not contain any interior lattice point. Hence, $r$ lies in the the interior of $P$. Similarly, for $T_{2}= \conv\{(0,0),(-1,-k),(0,1)\}$, we obtain that the point $r' = (-1,-k+1)$ lies in the interior of $P$. Similar to the arguments used in the case of heavy cycle with two loops, we realize that when we invoke convexity of $P$ along with the condition that $r$ and $r'$ are interior points of $P$, then we obtain that the sub polygon realizing the subgraph $G_{2}$ lies between the lines $x=1$ and $x=-1$ illustrated in Figure \ref{fig:heavy-one-loop}, which implies that $P_{3}$ is hyperelliptic.

We now move on to show that $g(P_{3}) \leq 3$. We proceed by contradiction and assume $g(P')\geq 4$. By our assumption, $(0,-3)$ is an interior lattice point of $P_{3}$. We notice that the point $p_1 = (1,0)$, which is a boundary point, and $r = (1,1)$, which is an interior lattice point, both lie on the line $x=1$. We now consider the possibilities for the point $(1,2)$; either $(1,2)\in P_1$, in which case $(1,2)\in \partial P_1$ or the boundary edge at $w$ passes through a point in the open interval $((1,2),(1,1))$. Also, no point in  $\partial P_{1}$ can be present on the line $y=3x-3$ because $(0,-3)$ is an interior lattice point. Hence, we conclude that $P_{1}$ is contained in the triangle $\conv\{p_1,w,(1,3)\}$. However, this triangle is not valid because $(1,3)$ has been excluded; see Figure~\ref{fig:cycle_two_loops}. This provides the desired contradiction, and thus $g(P') \leq 3$. 

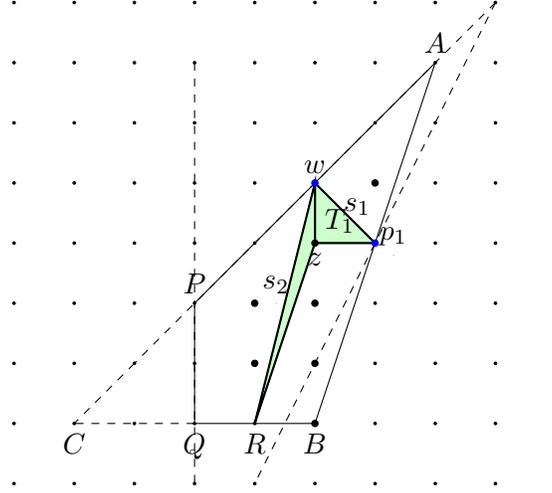
\begin{figure}[H]\centering
    \begin{tikzpicture}[scale=0.8]
    \coordinate (z) at (0,0);
    \coordinate (p1) at (1,0);
    \coordinate (q1) at (0,1);
    \coordinate (p2) at (-1,-5);
    \coordinate (q2) at (0,1);
    \coordinate (u1) at (1,1);
    \coordinate (v1) at (-1,-4);
    \coordinate (k1) at (0,-1);
    \coordinate (k2) at (0,-2);
    \coordinate (k3) at (0,-3);
    \coordinate (k4) at (-1,-3);
    \coordinate (k5) at (-1,-1);
    \coordinate (k6) at (-1,-2);

    \filldraw[triangle] (z) -- (p1) -- (q1) -- cycle;
    \filldraw[triangle] (z) -- (q2) -- (k4) -- cycle;
    
    \node[above right] at (z) {$T_1$};
    \filldraw[lattice] (z) circle (1.5pt);
    \draw[] (0,0) -- (0,1);
    \draw[] (0,0) -- (1,0);
    \draw[] (0,1) -- (1,0);
    \draw[] (0,1) -- (-1,-3);
    \draw[] (0,0) -- (-1,-3);
    \draw[] (2,3) -- (-2,-1);
    \draw[] (-2,-3) -- (-2,-1);
    \draw[] (-2,-3) -- (0,-3);
    \draw[] (0,-3) -- (2,3);
    \draw[dashed] (3,4) -- (-1,-4);
    \draw[dashed] (-3,-2) -- (-4,-3);
    \draw[dashed] (-2,-3) -- (-4,-3);
    \draw[dashed] (3,4) -- (-3,-2);
    \draw[dashed] (-2,3) -- (-2,-4);
    \fill[black] (0,0) circle (.000005cm) node[align=right,   below]{$z$};
    \fill[black] (2,3) circle (.000005cm) node[align=right,   above]{$A$};
    \fill[black] (0,-3) circle (.000005cm) node[align=right,   below]{$B$};
    \fill[black] (-4,-3) circle (.000005cm) node[align=right,   below]{$C$};
    \fill[black] (-2,-1) circle (.000005cm) node[align=left,   above]{$P$};
    \fill[black] (-2,-3) circle (.000005cm) node[align=left,   below]{$Q$};
    \fill[black] (-1,-3) circle (.000005cm) node[align=left,   below]{$R$};
    \fill[black] (0,1) circle (.000005cm) node[align=right,   above]{$w$};
    \fill[black] (0.7,0.3) circle (.000005cm) node[align=right,   above]{$s_{1}$};
    \fill[black] (-0.65,-1) circle (.000005cm) node[align=right,   above]{$s_{2}$};
    \fill[black] (1.3,-0.2) circle (.0000005cm) node[align=left,   above]{$p_1$};
    \foreach \x in {-5,...,4}{
    	\foreach \y in {-4,...,4}{
    		\filldraw[lattice] (\x,\y) circle (0.6pt);
    	}
    }
    \filldraw[boundary] (p1) circle (1.5pt);
    \filldraw[boundary] (q1) circle (1.5pt);
    \filldraw[lattice] (u1) circle (1.5pt);
    \filldraw[lattice] (k1) circle (1.5pt);
    \filldraw[lattice] (k2) circle (1.5pt);
    \filldraw[lattice] (k3) circle (1.5pt);
    \filldraw[lattice] (k5) circle (1.5pt);
    \filldraw[lattice] (k6) circle (1.5pt);
    \end{tikzpicture}
    \caption{This illustrates the case when $g=6$ and $g(P_{2}) = 2$}
    \label{fig:cycle_one_loop_g=6}
  \end{figure}

Now we show that $g(P) \leq 7$. We know that the points $r = (1,1)$ and $(0,-2)$ are interior points of $P$. Hence, no point in  $\partial P_{1}$ can be present on the line $y=2x-2$ because $(0,-2)$ is an interior lattice point. Again, $(1,2)$ is either in $\partial P$ or the boundary edge at $w$ passes through a point in the open interval $((1,2),(1,1))$. Therefore, we conclude that $P_{1}$ =$\conv\{p_{1},w, (2,3)\}$. Now, we try to find the maximal subpolygon $P_{2}$ (in terms of area)  that we can obtain given the above constraints. We realize that the maximal polygon in this case ${P_{2}}^{max} = \conv \{(0,1),(-1,-3),(-4,-3)\}$, which is shown as the triangle $wCR$ in Figure \ref{fig:cycle_one_loop_g=6}.

With the general statement proven, we now consider the specific case of genus $g = 6$. In the case that $g(P_{2}) = 2$, $P_{2}$ is hyperelliptic since all lattice polygons with $g \leq 2$ are hyperelliptic \cite{JB15}. We again proceed by constructing the maximal polygon $P$ in this case. Since $g(P_{2}) = 2$ and $P_{2}$ is hyperelliptic, we realize that the point $(-1,-1)$ is an interior point of $P$, because if it not in $P$ or is in $\partial P$, then it contradicts $g(P_{2}) = 2$. This implies that for a point $p=(x,y) \in P$, $x \geq -2$. This implies that the maximal polygon $P_{2}$ in this case is ${P_{2}}^{\text{max}} = \conv\{w,(-2,-1),(-2,-3),(-1,-3)\}$, which is depicted in the Figure \ref{fig:cycle_one_loop_g=6} as the quadrilateral $wPQR$. We realize that for $P_{2}$ and all its sub polygons, none of the corresponding triangulations $\Delta_{2}$ can posses a non trivial split edge, such that it divides $P_{2}$ into two sub polygons each of which has a positive genus. Hence, the proof.

\end{proof}

We state the forbidden criterion we obtain as a result of Theorem \ref{thm:heavy-one-loop},

\begin{corollary}
  Let $G$ be a planar trivalent graph of genus $g \geq 8$ such that $G$ has a heavy cycle with one loop. Then, $G$ cannot be tropically planar.
\end{corollary}

\begin{remark}
 Figure \ref{fig:g=6_realizable} shows the unimdoular triangulation which realizes a graph with a heavy cycle and one loop. Notice, that $G_{2}$ has no cut edges and hence $P_{2}$ has no nontrivial split.
\end{remark}

\begin{figure}[th]\centering
  \begin{tikzpicture}[scale=0.9]
    \draw[] (0,0) -- (1,0);
    \draw[] (0,0) -- (0,1);
    \draw[] (0,1) -- (1,0);
    \draw[] (1,0) -- (1,1);
    \draw[] (1,0) -- (1,2);
    \draw[] (1,1) -- (2,3);
    \draw[] (1,0) -- (2,3);
    \draw[] (0,1) -- (2,3);
    \draw[] (1,0) -- (2,3);
    \draw[] (0,1) -- (1,1);
    \draw[] (0,1) -- (2,3);
    \draw[] (0,1) -- (-1,-3);
    \draw[] (0,0) -- (-1,-3);
    \draw[] (0,0) -- (0,-1);
    \draw[] (0,-1) -- (1,0);
    \draw[] (1,0) -- (0,-2);
    \draw[] (0,-1) -- (0,-2);
    \draw[] (1,0) -- (0,-3);
    \draw[] (0,-2) -- (0,-3);
    \draw[] (0,-1) -- (-1,-3);
    \draw[] (0,-2) -- (-1,-3);
    \draw[] (0,-3) -- (-1,-3);
    \draw[] (0,1) -- (-2,-1);
    \draw[] (-2,-1) -- (-2,-3);
    \draw[] (-2,-3) -- (-1,-3);
    \draw[] (-1,-3) -- (-1,-2);
    \draw[] (-1,-2) -- (0,1);
    \draw[] (-1,-2) -- (-1,-1);
    \draw[] (-1,-2) -- (0,1);
    \draw[] (-1,-1) -- (0,1);
    \draw[] (-1,-1) -- (-1,0);
    \draw[] (-1,-1) -- (-2,-1);
    \draw[] (-1,-1) -- (-2,-2);
    \draw[] (-1,-2) -- (-2,-2);
    \draw[] (-1,-2) -- (-2,-3);
    \fill[black] (1,1) circle (.07cm);
    \fill[black] (0,0) circle (.07cm);
    \fill[black] (0,1) circle (.07cm);
    \fill[black] (1,1) circle (.07cm);    
    \fill[black] (1,2) circle (.07cm);
    \fill[black] (2,3) circle (.07cm);
    \fill[black] (0,-1) circle (.07cm);
    \fill[black] (0,-2) circle (.07cm);
    \fill[black] (-1,0) circle (.07cm);
    \fill[black] (-1,-1) circle (.07cm);
    \fill[black] (-1,-2) circle (.07cm);
    \fill[black] (-1,-3) circle (.07cm);
    \fill[black] (-2,-1) circle (.07cm);
    \fill[black] (-2,-2) circle (.07cm);
    \fill[black] (-2,-3) circle (.07cm);
    \fill[black] (1,0) circle (.07cm);
    \fill[black] (0,-3) circle (.07cm);
  \end{tikzpicture}
  \hspace{2cm}
   \begin{tikzpicture}[scale=0.7]
    \draw[] (0,0) -- (0,-2);
    \draw[] (0,-2) -- (2,-2);
    \draw[] (2,0) -- (2,-2);
    \draw[] (0,-4) -- (2,-4);
    \draw[] (0,-4) -- (0,-2);
    \draw[] (2,-4) -- (2,-2);
    \draw[] (2,-4) -- (2,-6);
    \draw[] (0,-4) -- (0,-5);
    \draw[] (0,-5) -- (-1,-7);
    \draw[] (0,-5) -- (1,-7);
    \draw (2,-6.5) circle (0.5cm);
    \draw (0,-7) ellipse (1cm and 0.2cm);
    \draw (1,0) ellipse (1cm and 0.2cm);
    \fill[black] (0,-4) circle (.1cm) node[align=left,   above]{};
    \fill[black] (0,-2) circle (.1cm) node[align=left,   above]{};
    \fill[black] (0,0) circle (.1cm) node[align=left,   above]{};
    \fill[black] (0,-5) circle (.1cm) node[align=left,   above]{};
    \fill[black] (-1,-7) circle (.1cm) node[align=left,   above]{};
    \fill[black] (1,-7) circle (.1cm) node[align=left,   above]{};
    \fill[black] (2,0) circle (.1cm) node[align=left,   above]{};
    \fill[black] (2,-2) circle (.1cm) node[align=left,   above]{};
    \fill[black] (2,-4) circle (.1cm) node[align=left,   above]{};
    \fill[black] (2,-6) circle (.1cm) node[align=left,   above]{};
  \end{tikzpicture}
  \caption{A unimodular triangulation of genus six with $g(P_{2}) = 2$ (left), corresponding skeleton with a heavy cycle with one loop with $G_{2}$ that does not have a cut edge (right)}
  \label{fig:g=6_realizable}
\end{figure}
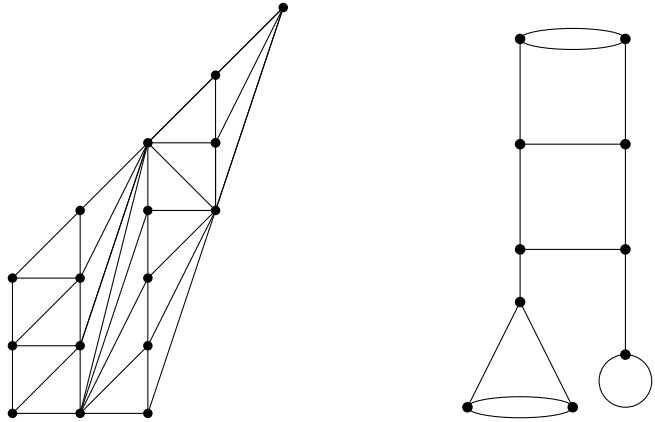

\section{Double heavy cycle}

We now define a \emph{double heavy cycle},

\begin{definition}\label{def:double_heavy}
  We say that cycles $C_{1},C_{2}$ in a planar graph $G$ are \emph{double heavy}, if
  \begin{enumerate}
  \item it has two nodes, $v_1$ and $v_2$, such that $v_i$ is incident with a cut edge $e_i$ connecting $v_i$ with a subgraph $G_i$ of positive genus;
  \item and there is a third subgraph, $G_3$, also of positive genus, which shares at least one node with the cycles $C_{1}$ and $C_{2}$; cf.\ Figure~\ref{fig:double_heavy}.

  \end{enumerate}
\end{definition}

\begin{figure}[H]\centering
  \begin{tikzpicture}
    \fill[gray!40!white] (0.5,-0.5) rectangle (2.5,-1.5);
    \fill[gray!40!white] (-2.5,2.5) rectangle (-0.5,1.5);
    \fill[gray!40!white] (3.5,2.5) rectangle (5.5,1.5);
    \draw[] (0,0) -- (-0.5,0.5);
    \draw[] (0,0) -- (0.5,-0.5);
    \draw[] (-0.5,0.5) -- (3.5,0.5);
    \draw[] (-1.5,1.5) -- (-0.5,0.5);
    \draw[] (3.5,0.5) -- (4.5,1.5);
    \draw[] (0.5,-0.5) -- (2.5,-0.5);
    \draw[] (0.5,-0.5) -- (0.5,-1.5);
    \draw[] (2.5,-0.5) -- (2.5,-1.5);
    \draw[] (2.5,-1.5) -- (0.5,-1.5);
    \draw[] (2.5,-0.5) -- (3,0);
    \draw[] (3,0) -- (3.5,0.5);
    \draw[] (5.5,1.5) -- (3.5,1.5);
    \draw[] (5.5,2.5) -- (3.5,2.5);
    \draw[] (5.5,1.5) -- (5.5,2.5);
    \draw[] (3.5,1.5) -- (3.5,2.5);
    \draw[] (-2.5,1.5) -- (-0.5,1.5);
    \draw[] (-2.5,2.5) -- (-0.5,2.5);
    \draw[] (-2.5,1.5) -- (-2.5,2.5);
    \draw[] (-0.5,1.5) -- (-0.5,2.5);
    \draw[] (1.6,0.5) -- (1.6,-0.5);
    \fill[black] (-0.5,0.5) circle (.09cm) node[align=right, above]{};
    \fill[black] (-0.3,0.5) circle (.000001cm) node[align=right, above]{$v_{2}$};
    \fill[black] (3.5,0.5) circle (.09cm) node[align=right,   above]{$v_{1}\quad$};
    \fill[black] (1.6,0.5) circle (.09cm) node[align=right,   above]{$v\quad$};
    \fill[black] (1,-0.25) circle (.0001cm) node[align=left,   above]{$C_{2}\quad$};
    \fill[black] (2.5,-0.25) circle (.0001cm) node[align=left,   above]{$C_{1}\quad$};
    \fill[black] (-1.1,0.7) circle (.0001cm) node[align=left,   above]{$e_{2}\quad$};
    \fill[black] (4.5,0.6) circle (.0001cm) node[align=left,   above]{$e_{1}\quad$};
    \fill[black] (1.65,-1.25) circle (.0001cm) node[align=left,   above]{$G_{3}\quad$};
    \fill[black] (-1.3,1.7) circle (.0001cm) node[align=left,   above]{$G_{2}\quad$};
    \fill[black] (4.7,1.7) circle (.0001cm) node[align=left,   above]{$G_{1}\quad$};
  \end{tikzpicture}
  \caption{Graph with double heavy cycles $C_{1}$ and $C_{2}$}
  \label{fig:double_heavy}
\end{figure}
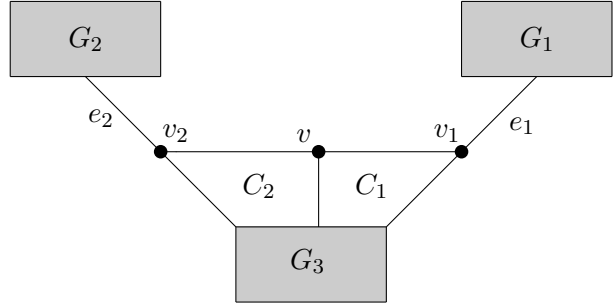

We fix some additional notation  concerning double heavy cycle. Let $T,T_{1}$ and $T_{2}$ denote the triangles in $\Delta$ dual to $v,v_{1}$ and $v_{2}$ in $G$ respectively. Let $z_{1}$ and $z_{2}$ represent the interior lattice points dual to $C_{1}$ and $C_{2}$. We also recall that the unimodular triangles in $\Delta$ can be categorized into degree one, two and three depending on the number of triangles adjacent to them in $\Delta$. Similar to Lemma \ref{lem:heavy}, we now prove a structural result concerning double heavy cycles, 

\begin{lemma}\label{lem:double_heavy}
  Suppose that $G$ has double heavy cycles $C_{1}$ and $C_{2}$ with cut edges $e_1$ and $e_2$ as in Figure~\ref{fig:heavy}. Then the triangles $T$ and $T_1$ in $\Delta$ share an edge $[z_{1},w]$, the triangles $T$ and $T_2$ in $\Delta$ share an edge $[z_{2},w]$ where $z_{1}$ is the interior lattice point dual to $C_{1}$ and $z_{2}$ is the interior lattice point dual to $C_{2}$. The split lines $S_1$ and $S_2$ intersect in $w$, which is a shared vertex between $T_{1}$ and $T_{2}$.
\end{lemma}

\begin{proof}
\sloppy{
Similar to the earlier proofs we fix a unimodular triangle $T_{1} =\conv \{(0,0),(0,1),(1,0)\}$, where $z_{1} = (0,0)$, with the split edge $s_{1}$ being between the points $(0,1)$ and $(1,0)$. Since the cycles $C_{1}$ and $C_{2}$ are adjacent in $G$, the interior lattice points dual to them would be hyperelliptic, i.e., they would lie on a line. Hence, $z_{2}$ would lie on the line $x=-1$. Also, by using Lemma \ref{lem:abz} on $T_{1}$ we infer that the point $(1,1)$ is an interior lattice point of $P$. 
}
  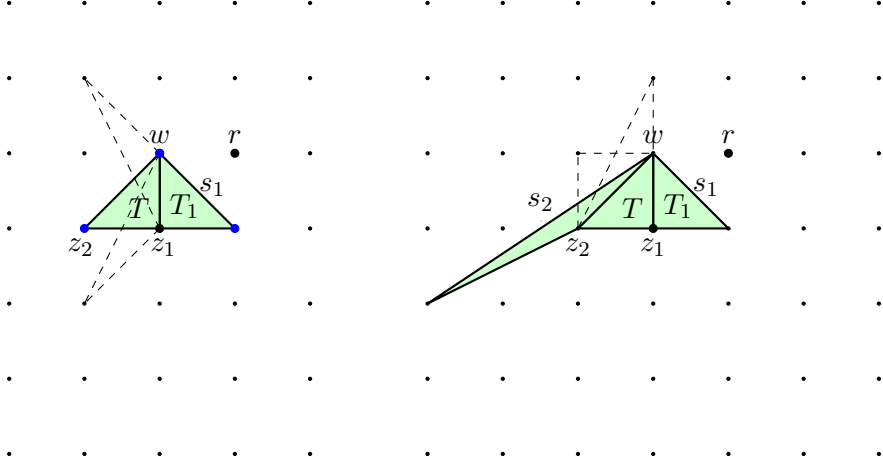
\begin{figure}\centering
    \begin{tikzpicture}[scale=1]
      \coordinate (z) at (0,0);
      \coordinate (p1) at (1,0);
      \coordinate (q1) at (0,1);
      \coordinate (p2) at (0,1);
      \coordinate (q2) at (-1,0);
      \coordinate (u1) at (1,1);
      \coordinate (v1) at (-1,-1);
      \coordinate (k1) at (0,-1);
      
      \filldraw[triangle] (z) -- (p1) -- (q1) -- cycle;
      \filldraw[triangle] (z) -- (p2) -- (q2) -- cycle;
      
      \node[above right] at (z) {$T_1$};
      \node[above left] at (z) {$T$};
      \filldraw[lattice] (z) circle (1.5pt);
      \draw[] (0,0) -- (0,1);
      \draw[] (0,0) -- (1,0);
      \draw[] (0,1) -- (1,0);
      \draw[dashed] (0,0) -- (-1,2);
      \draw[dashed] (0,1) -- (-1,2);
      \draw[dashed] (0,0) -- (-1,-1);
      \draw[dashed] (0,1) -- (-1,-1);
      \fill[black] (0.05,0) circle (.000005cm) node[align=right,   below]{$z_{1}$};
      \fill[black] (-1.05,0) circle (.000005cm) node[align=right,   below]{$z_{2}$};
      \fill[black] (0,1) circle (.000005cm) node[align=right,   above]{$w$};
      \fill[black] (1,1) circle (.000005cm) node[align=right,   above]{$r$};
      \fill[black] (0.7,0.3) circle (.000005cm) node[align=right,   above]{$s_{1}$};
      \foreach \x in {-2,...,2}{
        \foreach \y in {-3,...,3}{
          \filldraw[lattice] (\x,\y) circle (0.6pt);
        }
      }
      \filldraw[boundary] (p1) circle (1.5pt);
      \filldraw[boundary] (q1) circle (1.5pt);
      \filldraw[boundary] (p2) circle (1.5pt);
      \filldraw[boundary] (q2) circle (1.5pt);
      \filldraw[lattice] (u1) circle (1.5pt);
    \end{tikzpicture}
    \hspace{1.25cm} 
    \begin{tikzpicture}[scale=1]
    \coordinate (z) at (0,0);
    \coordinate (p1) at (1,0);
    \coordinate (q1) at (0,1);
    \coordinate (p2) at (-1,-5);
    \coordinate (q2) at (0,1);
    \coordinate (u1) at (1,1);
    \coordinate (v1) at (-1,-4);
    \coordinate (k1) at (0,-1);
    \coordinate (k2) at (0,-2);
    \coordinate (k3) at (0,-3);
    \coordinate (p2) at (-1,0);
    \coordinate (p3) at (-3,-1);
    
    \filldraw[triangle] (z) -- (p1) -- (q1) -- cycle;
    \filldraw[triangle] (z) -- (p2) -- (q1) -- cycle;
    \filldraw[triangle] (q1) -- (p3) -- (p2) -- cycle;
    
    \node[above right] at (z) {$T_1$};
    \node[above left] at (z) {$T$};
    \filldraw[lattice] (z) circle (1.5pt);
    \draw[] (0,0) -- (0,1);
    \draw[] (0,0) -- (1,0);
    \draw[] (0,1) -- (1,0);
    \draw[] (0,1) -- (-3,-1);
    \draw[] (-1,0) -- (-3,-1);

    \draw[dashed] (-1,1) -- (0,1);
    \draw[dashed] (-1,0) -- (-1,1);
    \draw[dashed] (-1,0) -- (0,2);
    \draw[dashed] (0,2) -- (0,1);
    \fill[black] (0,0) circle (.000005cm) node[align=right,   below]{$z_{1}$};
    \fill[black] (-1,0) circle (.000005cm) node[align=right,   below]{$z_{2}$};
    \fill[black] (0,1) circle (.000005cm) node[align=right,   above]{$w$};
    \fill[black] (1,1) circle (.000005cm) node[align=right,   above]{$r$};
    \fill[black] (0.7,0.3) circle (.000005cm) node[align=right,   above]{$s_{1}$};
     \fill[black] (-1.5,0.1) circle (.000005cm) node[align=left,   above]{$s_{2}$};
    \foreach \x in {-3,...,3}{
    	\foreach \y in {-3,...,3}{
    		\filldraw[lattice] (\x,\y) circle (0.6pt);
    	}
    }
    \filldraw[lattice] (u1) circle (1.5pt);
    \end{tikzpicture}
    \caption{Possibilities for the face $T$ (left) and an example with the faces $T_{2}$, $T$ and $T_{1}$ in $\Delta$ for $q=0$ (right) }
    \label{fig:double_heavy_lemma}
  \end{figure}
 
We now show that $T$ and $T_{1}$ share an edge. We assume to the contrary, that there exists at least one face $T'$ in between $T$ and $T_{1}$ which is either a degree two triangular face or a degree three triangular face adjacent to a degree one face. In both these cases, after deletion of degree one faces, $T'$ is a degree two triangular face. We consider the polygon $P'$ obtained from $P$ by recursive deletion of degree one faces. If $G$ is a graph realizable by $P$, then it is also realizable by $P'$. Henceforth, we consider the triangulation in $P'$. Given the structure of $G$ and our choice of $T_{1}$, we realize that the common edge between $T_{1}$ and $T'$ would be $[(0,0),(0,1)]$. Using unimodularity of $T'$ we realize that the third vertex of $T'$ lies on the line $x=-1$. Invoking convexity, presence of $T_{1}$ in $P'$ and the point $(1,1)$ being an interior lattice point, we observe that any point $(-1,q)$ with $q=1$ cannot be a vertex of $T'$. Therefore we consider the third vertex of $T'$ to be $(-1,q), q \leq 0$. Since $T'$ is a degree two triangular face, this would imply that one of its edge is along $\partial P'$, which gives us a contradiction because then it can only be adjacent to $T$ and $T_{1}$ and not to any other triangle, but it has to be adjacent to at least one other triangle from the heavy component which realizes $G_{3}$. Hence, we conclude that $T$ and $T_{1}$ share an edge in his case. Since, $v$ is symmetric with respect to $v_{1}$ and $v_{2}$ in $G$, therefore $T$ is symmetric with respect to $T_{1}$ and $T_{2}$. Hence, we also conclude that $T$ and $T_{2}$ share an edge, illustrated in Figure \ref{fig:double_heavy_lemma}.

We now are left to show that $T_{1}$ and $T_{2}$ share a vertex. We realize that given $T_{1} = \conv \{(0,0),(1,0),(0,1)\}$ and $T = \{(0,0), (0,1), (-1,q)\}, \> q \leq 0$, the edge that is shared by $T$ and $T_{2}$ is the edge between the points $(-1,q)$ and $(0,1)$. When we consider all possible candidates for the third vertex $(\alpha,\beta)$ of $T_{2}$, we realize that $\alpha < 0$. Also, if $\beta > 1$, then this contradicts convexity of $P$, at the point $(0,1)$. Hence, $\beta \leq 1$. If $\beta = 1$, this would imply that $P_{1}$ is squeezed between the line $y = 0$ and $y =1$, which has no interior lattice points. Therefore, $\beta \leq 0$. We obtain that for all such values of  $\alpha$ and $\beta$ the split edge $s_{2}$ is of the form $[(\alpha,\beta),(0,1)]$. This implies that $T_{1}$ and $T_{2}$ share the vertex $w = (0,1)$, illustrated in Figure \ref{fig:double_heavy_lemma}. Hence, the proof.
\end{proof}

We now define a graph having a double heavy cycle with two loops,

\begin{definition}
 We say that a connected trivalent planar graph $G$ has a \emph{double heavy cycle with two loops} if it has the form as described in Figure \ref{fig:double_heavy_two_loops}, where the shaded region $G_{3}$ a subgraph of positive genus.  
\end{definition}

\begin{figure}[H]\centering
  \begin{tikzpicture}
    \coordinate (v2) at (-0.5,0.5);
	\coordinate (v1) at (3.5,0.5);
	\coordinate (v) at (1.6,0.5);
	\coordinate (u2) at (-1.5,1.5);
	\coordinate (u1) at (4.5,1.5);
	
	\fill[gray!40!white] (0.5,-0.5) rectangle (2.5,-1);
	\draw[] (0,0) -- (-0.5,0.5);
	\draw[] (-0.5,0.5) -- (3.5,0.5);
	\draw[] (0,0) -- (0.5,-0.5);
	\draw[] (-1.5,1.5) -- (-0.5,0.5);
	\draw[] (3.5,0.5) -- (4.5,1.5);
	\draw[] (-1.85,1.85) circle (0.5cm);
	\draw[] (0.5,-0.5) -- (2.5,-0.5);
	\draw[] (0.5,-0.5) -- (0.5,-1);
	\draw[] (2.5,-0.5) -- (2.5,-1);
	\draw[] (2.5,-1) -- (0.5,-1);
	\draw[] (2.5,-0.5) -- (3,0);
	\draw[] (3,0) -- (3.5,0.5);
	\draw[] (1.6,0.5) -- (1.6,-0.5);
	\draw[] (4.85,1.85) circle (0.5cm);
	
    \node at ($(v2)$) [above right=-0.1pt] {$v_{2}$};
	\node at ($(v1)$) [above left=-0.1pt] {$v_{1}$};
	\node at ($(v)$) [above left=-0.05pt] {$v$};
	
	\fill[black] (-0.5,0.5) circle (.09cm);
	\fill[black] (3.5,0.5) circle (.09cm);
	\fill[black] (1.6,0.5) circle (.09cm);
	\fill[black] (-1.25,0.7) circle (.0001cm) node[align=left,   above]{$e_{2}$};
	\fill[black] (4.3,0.7) circle (.0001cm) node[align=left,   above]{$e_{1}$};
	\fill[black] (4.5,1.5) circle (.1cm) node[align=left,   above]{};
	\fill[black] (-1.5,1.5) circle (.1cm) node[align=left,   above]{};
	\fill[black] (1,-0.2) circle (.0001cm) node[align=left,   above]{$C_{2}\quad$};
	\fill[black] (2.6,-0.2) circle (.0001cm) node[align=left,   above]{$C_{1}\quad$};	
	\fill[black] (1.65,-1.05) circle (.0001cm) node[align=left,   above]{$G_{3}\quad$};
	\end{tikzpicture}
	\caption{Double heavy cycle with two loops}
	\label{fig:double_heavy_two_loops}
\end{figure}
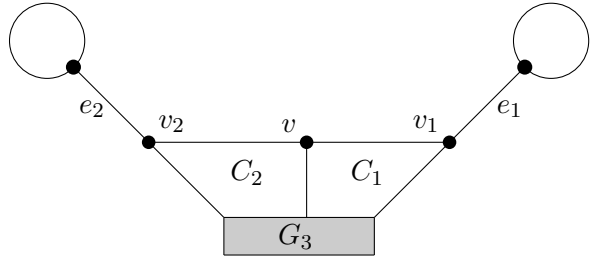

We now establish a structural result concerning double heavy cycles with two loops,

\begin{theorem}\label{thm:double_heavy_two_loops}
Suppose $G$ is a graph with double heavy cycles $C_{1}$ and $C_{2}$ with two loops with cut edges, $e_{1}$ and $e_{2}$, as in Figure \ref{fig:double_heavy_two_loops}. If $g(P)= 6$, then the interior lattice polygon of the heavy component, i.e., $\text{int}(P_{3})$ is a unit parallelogram. 
\end{theorem}

\begin{proof}

We fix the triangle $T_{1} =  \conv\{(0,0),(0,1),(1,0)\}$, where $z_{1} = (0,0)$  We know by Lemma \ref{lem:double_heavy} that $T= \conv \{(0,0), (0,1), (-1,q)\}, q \leq 0$ and the triangles $T,T_{1}$ and $T_{2}$ each share an edge amongst themselves. By applying Lemma \ref{lem:abz} on $T_{1}$ we obtain that $(1,1)$ is the unique interior lattice point of $P_{1}$. We see that in this case the interior lattice points of the heavy component are squeezed between the lines $y = (1-q)x+1$ and $y=3x-3$. Also, using convexity of $P$ at the point $(1,0)$, we conclude that $(-1,q-1)$ is an interior point of $P_{3}$ if $(0,-1) \in \text{int}(P_{3})$. We realize that for the case when  $g(P_{3}) =4$, the point $(0,-1) \in \text{int}(P_{3})$ and we need to show that both the lines $x=0$ and $x=-1$ each have one interior lattice point other than the points $(0,0)$ and $(-1,q)$. We assume the contrary, i.e., the line $x=0$ has two interior lattice points. If $-1 \geq q \leq 0$ then this contradicts convexity of $P$ at the point $(-1,q)$. If $q \leq -1 $, then we obtain more interior lattice points in $P_{3}$ which contradicts the fact that $g(P_{3}) = 4$. Hence, we conclude that for $g=6$ the interior lattice polygon of the heavy component, i.e., $\text{int}(P_{3})$ is a unit parallelogram.

  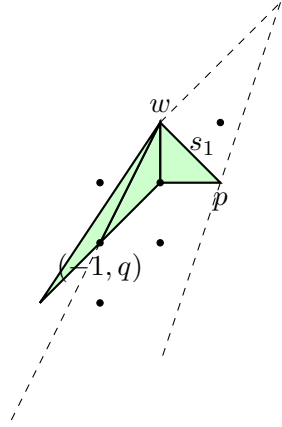
\begin{figure}[H]\centering
    \begin{tikzpicture}[scale=0.8]
    \coordinate (z) at (0,0);
    \coordinate (p1) at (1,0);
    \coordinate (q1) at (0,1);
    \coordinate (q2) at (0,1);
    \coordinate (u1) at (1,1);
    \coordinate (v1) at (-1,-4);
    \coordinate (k1) at (0,-1);
    \coordinate (k2) at (0,-2);
    \coordinate (k3) at (0,-3);
    \coordinate (p2) at (-1,-1);
    \coordinate (p3) at (-2,-2);
    \coordinate (q3) at (-1,-2);
    \coordinate (i1) at (-1,-1);
    \coordinate (i2) at (-1,0);
    \coordinate (i3) at (0,-1);
    
    \filldraw[triangle] (z) -- (p1) -- (q1) -- cycle;
    \filldraw[triangle] (z) -- (p2) -- (q1) -- cycle;
    \filldraw[triangle] (q1) -- (p3) -- (p2) -- cycle;
    
    \filldraw[lattice] (z) circle (1.5pt);
    \draw[] (0,0) -- (0,1);
    \draw[] (0,0) -- (1,0);
    \draw[] (0,1) -- (1,0);
    \draw[dashed] (0,1) -- (2,3);
    \draw[dashed] (-1,-1) -- (-2.5,-4);
    \draw[dashed] (2,3) -- (0,-3);


    \fill[black] (-1,-1) circle (.000005cm) node[align=right,   below]{$(-1,q)$};
    \fill[black] (0,1) circle (.000005cm) node[align=right,   above]{$w$};
    \fill[black] (1,0) circle (.000005cm) node[align=right,   below]{$p$};

    \fill[black] (0.7,0.3) circle (.000005cm) node[align=right,   above]{$s_{1}$};
    \filldraw[lattice] (u1) circle (1.5pt);
    \filldraw[lattice] (i3) circle (1.5pt);
    \filldraw[lattice] (i2) circle (1.5pt);
    \filldraw[lattice] (i1) circle (1.5pt);
    \filldraw[lattice] (q3) circle (1.5pt);
        
    \end{tikzpicture}
    \qquad
    \caption{The triangular faces $T,T_{1}$ and $T_{2}$ in a triangulation dual to a graph with double heavy cycle with two loops}
    \label{fig:double_heavy_theorem}
  \end{figure}

\end{proof}

\section{Conclusions}

We now furnish conclusions to what our results entail with respect to classification of skeletons of genus six. Firstly, we call the graph 'g' in Figure \ref{fig:genus6_uncategorized} as an \emph{enve-loop} graph (envelope + loop). We also recognize that the graphs 'a','b','c' and 'd' all have a heavy cycle with one loop. Out of these the heavy component in the graph 'a' is not hyperelliptic, hence it can be ruled out by Theorem \ref{thm:heavy-one-loop}. Also, graph 'b' does have a genus three hyperelliptic heavy component, although the orientation of the hyperelliptic points is not permissible according to Theorem \ref{thm:heavy-one-loop}, hence it can be ruled out as well. Graphs 'c' and 'd' both have permissible hyperelliptic heavy components, but both of them have a genus two component which has a cut edge, which would correspond to a nontrivial split in $P_{2}$ which is not possible by Theorem \ref{thm:heavy-one-loop}, hence these two graphs are also ruled out.

Subsequently, we realize that graphs 'e', 'f' and 'h' all have double heavy cycle with two loops. By Theorem \ref{thm:double_heavy_two_loops}, we know that the convex hull of the interior lattice points of the heavy component $P_{3}$ is a unit parallelogram, for a lattice polygon $P$ realizing a double haeavy cycle with two loops. Hence, graph 'f' is eliminated as the heavy component cannot be a parallelogram in this case. Let $C_{1},C_{2},C_{3}$ and $C_{4}$ be the four cycles in the subgraph dual to the heavy component $P_{3}$, within a realizable graph of $g=6$ with double heavy cycle with two loops. Given that the graph is trivalent, we realize that the subgraph cannot have two connected components each of genus two. Only possible connected components are of genus three and genus one. Since the interior points dual to $C_{i}'$s form a unit parallelogram, this implies that either the cycles $C_{1},C_{2}$ and $C_{3}$ share a vertex or the the cycles $C_{1},C_{2}$ and $C_{4}$ share a vertex. Hence, we conclude that for a realizable graph $G$ of genus six with a double heavy cycle with two loops, in the subgraph realizing the heavy component, at least three cycles should share a vertex. Therefore, with this characterization we can rule out the graphs 'e' and 'h'. Now we are ready to state the full charachterization of all tropical planar curves up to genus six, which can be seen as a generalization of Theorem 3 in \cite{Joswig2020},

\begin{theorem}\label{thm:classify_genus_six}
A trivalent planar graph $G$ other than the enve-loop graph, of genus $g \leq 6$, is tropically planar if and only if none of the following obstructions occur
\begin{enumerate}
    \item it contains a sprawling node, or
    \item it contains a sprawling triangle and $g \geq 5$, or
    \item it is crowded, or
    \item it is a TIE-fighter, or
    \item it has a heavy cycle with two loops such that the interior lattice points of the
heavy component do not align with the intersection of the two split lines, or
    \item for a cut edge $e$ in $G$, the connected components of $G \setminus \{e\}$, after smoothing out 2-valent vertices, are not tropically planar, or
    \item it has a heavy cycle with one loop such that either the interior lattice points of the
heavy component do not align with the intersection of the two split lines or the connected component with genus greater than one has a cut edge, or
    \item it has a double heavy cycle with two loops such that either the interior lattice points of the heavy component do not form a unit parallelogram or no three cycles in the heavy component share a vertex. 
\end{enumerate}
\end{theorem}

We know that as the genus increases almost all graphs are not troplanar \cite[Theorem 4.2 ]{M19}. Additionally, the class of troplanar graphs is not minor-closed \cite[Figure 6]{M19}. Hence, such classifications become increasingly difficult as the genus increases. One area of future exploration could be to finish the classification for the case of genus seven, however no list of unclassified graphs is available for it. Also, there has been recent progress in finding criteria which do not include a cut edge namely the criteria of \emph{big face graphs} defined in \cite{MT21}, the techniques of which can be used to define new criteria for higher genus.   

\bibliographystyle{siam}
\bibliography{biblio.bib}

\end{document}